\documentclass[12pt]{amsart}

\title[]{The inclusion relations of the countable models of set theory are all isomorphic}

\author[Hamkins]{Joel David Hamkins}
 \address[J.~D.~Hamkins]
         {Mathematics, Philosophy, Computer Science, The Graduate Center of The City University of New York,
         365 Fifth Avenue, New York, NY 10016 \&
         Mathematics, College of Staten Island of CUNY, Staten Island, NY 10314}
 \email{jhamkins@gc.cuny.edu}
 \urladdr{http://jdh.hamkins.org}
 \thanks{This work was supported by JSPS KAKENHI Grant Number 17H02263. The research project builds on our earlier paper~\cite{HamkinsKikuchi2016:Set-theoreticMereology}. This segment of the work began in Kyoto at the conference \emph{Mathematical Logic and its Applications}, organized by the second author and held at the Research Institute for Mathematical Sciences at Kyoto University in September 2016. The first author is grateful for the support of his participation there. Commentary concerning this paper can be made at \href{http://jdh.hamkins.org/inclusion-relations-are-all-isomorphic}{jdh.hamkins.org/inclusion-relations-are-all-isomorphic}.}

\author[Kikuchi]{Makoto Kikuchi}
 \address[M.~Kikuchi]
         {Graduate School of System Informatics, Kobe University, Rokkodai, Nada, Kobe 657-8501, Japan}
 \email{mkikuchi@kobe-u.ac.jp}
 \urladdr{http://www2.kobe-u.ac.jp/~mkikuchi/index-e.html}

\newtheorem{theorem}{Theorem}

\newtheorem*{maintheorem*}{Main Theorem}
\newtheorem*{maintheorems*}{Main Theorems}
\newtheorem{corollary}[theorem]{Corollary}
\newtheorem*{corollary*}{Corollary}
\newtheorem*{corollaries*}{Corollaries}

\newtheorem{lemma}[theorem]{Lemma}

\newtheorem{question}[theorem]{Question}
\newtheorem*{question*}{Question}

\newtheorem*{questions*}{Questions}
\newtheorem*{mainquestion*}{Main Question} 
\newtheorem*{openquestion*}{Open Question} 
\newtheorem{observation}[theorem]{Observation}

\newcommand{\QED}{\end{proof}}

\def\proclaim[#1]{{\bf #1}}
\def\BF#1.{{\bf #1.}}

\newcommand\Ersov{Er\v sov}

\newcommand{\Godel}{G\"odel}

\newcommand{\of}{\subseteq}

\newcommand{\sqof}{\sqsubseteq}

\newcommand{\set}[1]{\{\,{#1}\,\}}
\newcommand{\singleton}[1]{\left\{{#1}\right\}}

\newcommand{\restrict}{\upharpoonright} 
\newcommand{\satisfies}{\models}

\newcommand{\concat}{\mathbin{{}^\smallfrown}}

\newcommand{\union}{\cup}

\newcommand{\intersect}{\cap}

\newcommand{\smalllt}{\mathrel{\mathchoice{\raise2pt\hbox{$\scriptstyle<$}}{\raise1pt\hbox{$\scriptstyle<$}}{\raise0pt\hbox{$\scriptscriptstyle<$}}{\scriptscriptstyle<}}}
\newcommand{\smallleq}{\mathrel{\mathchoice{\raise2pt\hbox{$\scriptstyle\leq$}}{\raise1pt\hbox{$\scriptstyle\leq$}}{\raise1pt\hbox{$\scriptscriptstyle\leq$}}{\scriptscriptstyle\leq}}}

\newcommand{\boolval}[1]{\mathopen{\lbrack\!\lbrack}\,#1\,\mathclose{\rbrack\!\rbrack}}
\def\[#1]{\boolval{#1}}
\newbox\gnBoxA
\newdimen\gnCornerHgt
\setbox\gnBoxA=\hbox{\tiny$\ulcorner$}
\global\gnCornerHgt=\ht\gnBoxA
\newdimen\gnArgHgt
\def\gcode #1{%
\setbox\gnBoxA=\hbox{$#1$}%
\gnArgHgt=\ht\gnBoxA%
\ifnum     \gnArgHgt<\gnCornerHgt \gnArgHgt=0pt%
\else \advance \gnArgHgt by -\gnCornerHgt%
\fi \raise\gnArgHgt\hbox{\tiny$\ulcorner$} \box\gnBoxA %
\raise\gnArgHgt\hbox{\tiny$\urcorner$}}
\newcommand{\UnderTilde}[1]{{\setbox1=\hbox{$#1$}\baselineskip=0pt\vtop{\hbox{$#1$}\hbox to\wd1{\hfil$\sim$\hfil}}}{}}
\newcommand{\Undertilde}[1]{{\setbox1=\hbox{$#1$}\baselineskip=0pt\vtop{\hbox{$#1$}\hbox to\wd1{\hfil$\scriptstyle\sim$\hfil}}}{}}
\newcommand{\undertilde}[1]{{\setbox1=\hbox{$#1$}\baselineskip=0pt\vtop{\hbox{$#1$}\hbox to\wd1{\hfil$\scriptscriptstyle\sim$\hfil}}}{}}
\newcommand{\UnderdTilde}[1]{{\setbox1=\hbox{$#1$}\baselineskip=0pt\vtop{\hbox{$#1$}\hbox to\wd1{\hfil$\approx$\hfil}}}{}}
\newcommand{\Underdtilde}[1]{{\setbox1=\hbox{$#1$}\baselineskip=0pt\vtop{\hbox{$#1$}\hbox to\wd1{\hfil\scriptsize$\approx$\hfil}}}{}}

\def\<#1>{\left\langle#1\right\rangle}

\newcommand{\ZFC}{{\rm ZFC}}
\newcommand{\ZF}{{\rm ZF}}

\newcommand\ZFfin{\ZF^{\neg\infty}}

\newcommand{\KM}{{\rm KM}}

\newcommand{\GBC}{{\rm GBC}}

\newcommand{\CH}{{\rm CH}}

\newcommand{\KP}{{\rm KP}}

\newcommand{\HF}{{\rm HF}}

\newcommand{\PA}{{\rm PA}}

\newcommand{\cell}[1]{\boxit{\hbox to 17pt{\strut\hfil$#1$\hfil}}}
\newcommand{\head}[2]{\lower2pt\vbox{\hbox{\strut\footnotesize\it\hskip3pt#2}\boxit{\cell#1}}}
\newcommand{\boxit}[1]{\setbox4=\hbox{\kern2pt#1\kern2pt}\hbox{\vrule\vbox{\hrule\kern2pt\box4\kern2pt\hrule}\vrule}}
\newcommand{\Col}[3]{\hbox{\vbox{\baselineskip=0pt\parskip=0pt\cell#1\cell#2\cell#3}}}
\newcommand{\tapenames}{\raise 5pt\vbox to .7in{\hbox to .8in{\it\hfill input: \strut}\vfill\hbox to
.8in{\it\hfill scratch: \strut}\vfill\hbox to .8in{\it\hfill output: \strut}}}
\newcommand{\Head}[4]{\lower2pt\vbox{\hbox to25pt{\strut\footnotesize\it\hfill#4\hfill}\boxit{\Col#1#2#3}}}
\newcommand{\Dots}{\raise 5pt\vbox to .7in{\hbox{\ $\cdots$\strut}\vfill\hbox{\ $\cdots$\strut}\vfill\hbox{\
$\cdots$\strut}}}
\usepackage[hidelinks]{hyperref}
\usepackage{needspace}
\usepackage{latexsym,amsfonts,amsmath,amssymb}
\usepackage{relsize}
\usepackage[rgb,dvipsnames]{xcolor}
\usepackage{tikz} 
\usetikzlibrary{arrows,hobby}
\usepackage{diagrams}\diagramstyle[tight,centredisplay,textflow]

\begin{document}

\begin{abstract}
 The structures $\<M,\of^M>$ arising as the inclusion relation of a countable model of sufficient set theory $\<M,\in^M>$, whether well-founded or not, are all isomorphic. These structures $\<M,\of^M>$ are exactly the countable saturated models of the theory of set-theoretic mereology: an unbounded atomic relatively complemented distributive lattice. A very weak set theory suffices, even finite set theory, provided that one excludes the $\omega$-standard models with no infinite sets and the $\omega$-standard models of set theory with an amorphous set. Analogous results hold also for class theories such as \Godel-Bernays set theory and Kelley-Morse set theory.
\end{abstract}

\maketitle

\section{Introduction}

\noindent
Set-theoretic mereology is the study of the inclusion relation $\of$ as it arises within set theory. In any set-theoretic context, with the set membership relation $\in$, one may define the corresponding inclusion relation $\of$ and investigate its properties. Thus, every model of set theory $\<M,\in^M>$ gives rise to a corresponding model of set-theoretic mereology $\<M,\of^M>$, the reduct to the inclusion relation.

In our previous article~\cite{HamkinsKikuchi2016:Set-theoreticMereology}, we identified exactly the complete theory of these mereological structures $\<M,\of^M>$. Namely, if $\<M,\in^M>$ is a model of set theory, even for extremely weak theories, including set theory without the infinity axiom, then the corresponding mereological reduct $\<M,\of^M>$ is an unbounded atomic relatively complemented distributive lattice. We call this the theory of set-theoretic mereology. By a quantifier-elimination argument that we give in~\cite{HamkinsKikuchi2016:Set-theoreticMereology}, partaking of Tarski's Boolean-algebra invariants and \Ersov's work on lattices, this theory is complete. 

After that work, we found it natural to inquire:

\begin{question}\label{Question.Which-models-arise?}
 Which models of set-theoretic mereology arise as the inclusion relation $\of$ of a model of set theory?
\end{question}

More precisely, given a model $\<M,\sqof>$ of set-theoretic mereology, under what circumstances can we place a binary relation $\in^M$ on $M$ in such a way that $\<M,\in^M>$ is a model of set theory and the inclusion relation $\of$ defined in $\<M,\in^M>$ is precisely the given relation $\sqof$? One can view this question as seeking a kind of Stone-style representation of the mereological structure $\<M,\sqof>$, because such a model $M$ would provide a representation of $\<M,\sqof>$ as a relative field of sets via the model of set theory $\<M,\in^M>$.

A second natural question was to wonder how much of the theory of the original model of set theory can be recovered from the mereological reduct.

\begin{question}\label{Question.What-does-inclusion-know?}
 If $\<M,\of^M>$ is the model of set-theoretic mereology arising as the inclusion relation $\of$ of a model of set theory $\<M,\in^M>$, what part of the theory of $\<M,\in^M>$ is determined by the structure $\<M,\of^M>$?
\end{question}

In the case of the countable models of \ZFC, these questions are completely answered by our main theorems. 

\begin{maintheorems*}\
 \begin{enumerate}
   \item All countable models of set theory $\<M,\in^M>\satisfies\ZFC$ have isomorphic reducts $\<M,\of^M>$ to the inclusion relation.
   \item The same holds for models of considerably weaker theories such as $\KP$ and even finite set theory $\ZFfin$, provided one excludes the $\omega$-standard models without infinite sets and the $\omega$-standard models having an amorphous set.
   \item These inclusion reducts $\<M,\of^M>$ are precisely the countable saturated models of set-theoretic mereology.
   \item Similar results hold for class theory: all countable models of \Godel-Bernays set theory have isomorphic reducts to the inclusion relation, and this reduct is precisely the countably infinite saturated atomic Boolean algebra.
 \end{enumerate}
\end{maintheorems*}

Specifically, in theorem~\ref{Theorem.Omega-saturated} we show that the mereological reducts $\<M,\of^M>$ of the models of sufficient set theory are always $\omega$-saturated, and from this it follows on general model-theoretic grounds (corollary~\ref{Corollary.All-isomorphic}) that they are all isomorphic, establishing statements (1) and (2). So a countable model $\<M,\sqof>$ of set-theoretic mereology arises as the inclusion relation of a model of sufficient set theory if and only if it is $\omega$-saturated (corollary~\ref{Corollary.Which-models-arise}), establishing (3) and answering question~\ref{Question.Which-models-arise?}. Consequently, in addition, the mereological reducts $\<M,\of^M>$ of the countable models of sufficient set theory know essentially nothing of the theory of the structure $\<M,\in^M>$ from which they arose, since $\<M,\of^M>$ arises equally as the inclusion relation of other models $\<M,\in^*>$ with any desired sufficient alternative set theory (corollary~\ref{Corollary.Alternative-theory}), a fact which answers question~\ref{Question.What-does-inclusion-know?}. Our analysis works with very weak set theories, even finite set theory $\ZFfin$, provided one excludes the $\omega$-standard models with no infinite sets and the $\omega$-standard models with an amorphous set, since the inclusion reducts of these models are not $\omega$-saturated. In section~\ref{Section.Uncountable-models} we prove that most of these results do not generalize to uncountable models, nor even to the $\omega_1$-like models, although theorem~\ref{Theorem.Saturated-models-are-realized} shows that every saturated model of set-theoretic mereology is realized as the inclusion relation of a model of any desired consistent set theory.

Our results have some affinity with the classical results in models of arithmetic concerned with the additive reducts of models of \PA. Restricting a model of set theory to the inclusion relation $\of$ is, after all, something like restricting a model of arithmetic to its additive part. Lipshitz and Nadel~\cite{LipshitzNadel1978:The-additive-structure-of-models-of-arithmetic} proved that a countable model of Presburger arithmetic (with $+$ only) can be expanded to a model of \PA\ if and only if it is computably saturated. We had hoped at first to prove a corresponding result for the mereological reducts of the models of set theory. In arithmetic, the additive reducts are not all isomorphic, since the standard system of the \PA\ model is fully captured by the additive reduct. Our main result for the countable models of set theory, however, turned out to be stronger than we had expected, since the inclusion reducts are not merely computably saturated, but fully $\omega$-saturated, and this is why they are all isomorphic. Meanwhile, Lipshitz and Nadel point out that their result does not generalize to uncountable models of arithmetic, and similarly ours also does not generalize to uncountable models of set theory (see section~\ref{Section.Uncountable-models}). Another instance of the general phenomenon is known for real-closed fields, since results in~\cite{DAquinoKnightStarchenko2010:Real-closed-fields-and-models-of-PA,DAquinoKnightStarchenko2010:Corrigendum-to-Real-closed-fields-and-models-of-PA} show that a countable real closed field has an integer part that is a model of \PA\ just in case it is either Archimedean or computably saturated.

\break
\section{Expressive power of types in set-theoretic mereology}

Let us begin our analysis by observing that every model of set-theoretic mereology $\<M,\sqof>$ can be represented as a relative field of sets, that is, a collection of sets closed under intersection, union and relative complement. This can be seen simply by identifying every object in $M$ with the set of atoms below it, since one may readily verify that this representation respects the lattice structure of $\<M,\sqof>$. Therefore, allow us freely to use a set-theoretic terminology and notation in set-theoretic mereology, referring to the lattice operations as union, intersection and relative complement.

We shall now clarify the exact expressive power of types in set-theoretic mereology.

\begin{lemma}\label{Lemma.Types}
 If $p(a_1,\ldots,a_n)$ is a complete $n$-type in the language of set-theoretic mereology, then $p(a_1,\ldots,a_n)$ is equivalent over the theory of set-theoretic mereology to the assertions stating for each cell in the Venn diagram of the variables that it has some specific finite size or that it is infinite.
$$
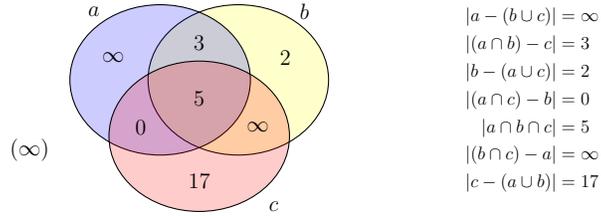
\begin{tikzpicture}[xscale=1.2,every node/.style={scale=.7}]
\draw[fill=yellow, fill opacity=.2,label={[opacity=1]above:$x_1$}] (30:.5) circle (1);
\draw[fill=blue, fill opacity=.2] (150:.5) circle (1);
\draw[fill=red, fill opacity=.2] (-90:.5) circle (1);
\draw (0,0) node {$5$};
\draw (30:1.1) node {$2$};
\draw (150:1.1) node {$\infty$};
\draw (-90:1.1) node {$17$};
\draw (90:.75) node {$3$};
\draw (-30:.75) node {$\infty$};
\draw (210:.75) node {$0$};
\draw (135:1.65) node {$a$};
\draw (45:1.65) node {$b$};
\draw (-60:1.65) node {$c$};
\draw (200:2) node {$(\infty)$};
\end{tikzpicture}
\qquad\qquad
\raise 1.5cm\hbox{\resizebox{!}{1.3cm}{
   $\begin{split}
      |a-(b\union c)|  &=\infty  \\
      |(a\intersect b)-c|  &=3  \\
      |b-(a\union c)|  &=2  \\
      |(a\intersect c)-b|  &=0  \\
      |a\intersect b\intersect c| &=5 \\
      |(b\intersect c)-a|  &=\infty\\
      |c-(a\union b)|  &=17  \\
   \end{split}$
}}
$$
\end{lemma}

\begin{proof}
This is a consequence of the elimination of quantifiers argument from our previous paper~\cite[theorem~9]{HamkinsKikuchi2016:Set-theoreticMereology}. We proved that every assertion in the language of set-theoretic mereology is equivalent to a quantifier-free assertion in the language allowing the operations of union $\union$, intersection $\intersect$ and relative complement $x-y$ and the relations $|\tau|=n$, which assert that there are precisely $n$ atoms below $\tau$. It follows that a complete type $p(a_1,\ldots,a_n)$ must make such an assertion about every cell in the corresponding Venn diagram of those variables, and furthermore this information determines everything else that one can express about those variables in this language. (Note that in set-theoretic mereology, it follows from unboundedness that the exterior region, which is not represented by any term, must always be infinite.)
\end{proof}

Although the lemma shows that every type amounts in a sense to finitely many assertions about the cells in the Venn diagram, we are not claiming that every complete type is principal, because the assertion that a particular cell in the Venn diagram is infinite, as with the assertion $|a-(b\union c)|=\infty$ in the diagram above, is not expressible by a single formula in the language of set-theoretic mereology, but rather is expressible in the type as infinitely many assertions stating that that term has no particular finite size. Indeed, one cannot express in a single formula that a term is infinite, since in~\cite{HamkinsKikuchi2016:Set-theoreticMereology} we proved that $\<\HF,\of>$ is an elementary substructure of $\<V,\of>$, and the former mereological structure has no infinite sets, while the latter does.

The lemma implies that there are only countably many types in set-theoretic mereology, since in finitely many variables there are only finitely many cells in the Venn diagram and only countably many possible things to say about each cell. It follows on general model-theoretic grounds, using the omitting-types theorem, that there is therefore a \emph{prime} model, a model that embeds elementarily into all other models. In the case of set-theoretic mereology, this is the model consisting of all finite subsets of a fixed countable set, such as the case with the structure $\<\HF,\of>$ of hereditarily finite sets.

It is also an immediate consequence of lemma~\ref{Lemma.Types} that every computably saturated model of set-theoretic mereology is fully $\omega$-saturated, because the expressive power of the types is so limited: every complete type is logically equivalent to a computable type.

\bigskip\bigskip
\section{Saturated models of set-theoretic mereology}

Since lemma~\ref{Lemma.Types} provides a complete account of what one can say with a type in set-theoretic mereology, we can use this to give a necessary and sufficient criteria for a model of set-theoretic mereology to be $\omega$-saturated.

Let us refer to an element $u$ in a model of set-theoretic mereology as an \emph{infinite} element, if it is not the join of finitely many atoms. This concept is not expressible in the language of set-theoretic mereology, in light of the $\<\HF,\of>\prec\<V,\of>$ example, and so one should understand it externally as a definition in the model theory of set-theoretic mereology. In particular, if $\<M,\in>$ is a model of set theory, an element $u$ can be infinite in $\<M,\of^M>$ in this sense without $\<M,\in>$ necessarily thinking that $u$ is an infinite set; for example, perhaps $u$ is a nonstandard finite set in $M$.

\bigskip\goodbreak
\begin{theorem}\label{Theorem.Omega-saturation-criterion}
 A model of set-theoretic mereology $\<M,\sqof>$ is $\omega$-saturated if and only if
 \begin{enumerate}
   \item every infinite element of $M$ is the disjoint union of two infinite elements, and
   \item for every element $a\in M$, there is an infinite element $u\in M$ disjoint from $a$.
 \end{enumerate}
Equivalently, there are infinite elements and for every infinite element $u$ there is an element $x$ for which $u-x$, $u\intersect x$ and $x-u$ are each infinite.
$$
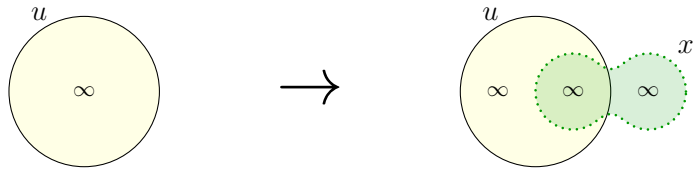
\begin{tikzpicture}[every node/.style={scale=.7}]
\draw[fill=yellow, fill opacity=.1] (-7,0) circle (1);
\draw (-7,0) node {$\infty$};
\draw (-7,0) +(120:1.2) node[scale=1.3] {$u$};
\draw (-4,0) node[scale=3] {$\to$};
\draw[fill=yellow, fill opacity=.1] (-1,0) circle (1);
\draw (-1,0) +(120:1.2) node[scale=1.3] {$u$};
\draw[use Hobby shortcut,line width=1pt, line cap=round, dash pattern=on 0pt off 2.7\pgflinewidth, closed=true,OliveGreen,fill=OliveGreen,fill opacity=.15] (1,0) .. (60:.5) .. (90:.3) .. (120:.5) .. (180:1) .. (240:.5) .. (-90:.3) .. (-60:.5);
\draw (180:1.5) node {$\infty$};
\draw (0:.5) node {$\infty$};
\draw (180:.5) node {$\infty$};
\draw (1,.6) node[scale=1.3] {$x$};
\end{tikzpicture}$$
\end{theorem}

\begin{proof}
($\to$) If $\<M,\sqof>$ is $\omega$-saturated and $u$ is infinite, then we may write down the type $p(x,u)$ expressing that $x\intersect u$ and $u-x$ are both infinite. This is expressible by the infinite list of assertions that $|x\intersect u|\neq k$ and $|u-x|\neq k$ for any finite $k$. Because $u$ has infinitely many atoms below it, every finite collection of these assertions is realized in the model, and so the type is finitely realized. By $\omega$-saturation, therefore, the whole type is realized by some object $v\in M$. So $u$ can be partitioned into $v\intersect u$ and $u-v$, both of which are infinite. Similarly, for any element $a\in M$, let $q(x,a)$ assert that $x-a$ is infinite, by means of the assertions $|x-a|\neq k$ for every finite $k$. Since the lattice is unbounded, this type is finitely realized, and so by $\omega$-saturation, the whole type is realized. So we have found an infinite element disjoint from $a$.

($\leftarrow$) Assume that $\<M,\sqof>$ is a model of set-theoretic mereology with the two stated properties. The main point is that this is sufficient to realize any given consistent type. To see this, suppose that $p(x,a_0,\dots,a_n)$ is a complete $1$-type in the language of mereology with finitely many parameters $a_i\in M$ and which is finitely realized in $\<M,\sqof>$. 
We want to show that the type is realized in $\<M,\sqof>$. By lemma~\ref{Lemma.Types}, the type is making assertions about the sizes of the various cells in the Venn diagram of $x$ and the parameters $a_i$. Note that the full Venn diagram for $x$ together with the parameters $a_i$ is obtained from the Venn diagram of the parameters $a_i$ alone by allowing $x$ to cut each cell in that diagram into two pieces. 
\begin{figure}[h]
\begin{tikzpicture}[scale=1.5,xscale=1.2,every node/.style={scale=.5}]
\draw[fill=yellow, fill opacity=.1,label={[opacity=1]above:$x_1$}] (30:.5) circle (1);
\draw[fill=blue, fill opacity=.1] (150:.5) circle (1);
\draw[fill=red, fill opacity=.1] (-90:.5) circle (1);
\draw[use Hobby shortcut,line width=1pt, line cap=round, dash pattern=on 0pt off 2.7\pgflinewidth, closed=true,OliveGreen,fill=OliveGreen,fill opacity=.15] (45:1) .. (90:.75) .. (150:1.1) .. (-150:1.8) ..  (-110:1) .. (-140:.7) .. (-90:0) .. (-45:.4);
\draw (90:.2) node {$3$};
\draw (-90:.3) node {$2$};
\draw (30:1.25) node {$1$};
\draw (30:.9) node {$1$};
\draw (140:1.2) node {$\infty$};
\draw (160:1) node {$\infty$};
\draw (-90:1.1) node {$13$};
\draw (-120:1.15) node {$4$};
\draw (90:.95) node {$3$};
\draw (65:.7) node {$0$};
\draw (-10:.7) node {$\infty$};
\draw (-40:.8) node {$56$};
\draw (200:.75) node {$0$};
\draw (235:.66) node {$0$};
\draw (210:1.5) node {$\infty$};
\draw (135:1.65) node[scale=1.3] {$a$};
\draw (45:1.65) node[scale=1.3] {$b$};
\draw (-60:1.65) node[scale=1.3] {$c$};
\draw (210:2) node[scale=1.3] {$x$};
\draw (0:3) node[scale=1.3] {\parbox[c]{6cm}{\centering A complete type $p(x,a,b,c)$ makes assertions about how $x$ splits the cells in the Venn diagram of $a$, $b$ and $c$ and how much of $x$ is outside $a\union b\union c$.}};
\end{tikzpicture}
\caption{A complete $1$-type $p(x,a,b,c)$ with three parameters.}\label{Figure.Complete-type}
\end{figure}
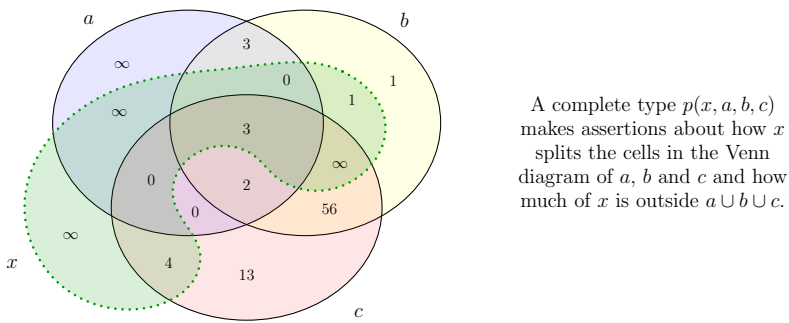
That is, the type $p(x,\vec a)$ is telling us how much to take from each cell in the Venn diagram of the parameters and how much to take from outside the union of the parameters (see~figure \ref{Figure.Complete-type}).

We claim that the type is realized, since we may simply assemble a realizing object $x$ by including the right number of objects from each cell, and this is possible under the assumptions that we have made about $\<M,\sqof>$. If the type asserts that $x\intersect\tau(\vec a)$ has some size $k$, for example, and that $\tau(\vec a)-x$ has size $r$, where $\tau(\vec a)$ is a term representing a cell in the Venn diagram of the parameters, then since the type is finitely realized, it follows that $\tau(\vec a)$ must have had size $k+r$ in $\<M,\sqof>$, and we may simply reserve $k$ of the objects from $\tau(\vec a)$ to place into $x$ and keep the others from that cell out of $x$. Similarly, if the type asserts that an infinite cell $\tau(\vec a)$ should be split in a certain way, including into two infinite pieces, or that there should be infinitely many elements of $x$ outside of $a_1\union\cdots\union a_n$, then these are also each possible by our assumption on $\<M,\sqof>$.

So for each cell in the Venn diagram, we may find a set $x$ satisfying the requirement that the type asserted for that cell, and the union of these local solutions, therefore, satisfies all the requirements of the type. So $\<M,\sqof>$ realizes every $1$-type $p(x,\vec a)$ with finitely many parameters that is consistent with its diagram, and so the structure is $\omega$-saturated.
\end{proof}

Since the models of set-theoretic mereology arising from a model of set theory generally exhibit the properties stated in theorem~\ref{Theorem.Omega-saturation-criterion}, it follows that they are $\omega$-saturated. There are some exceptions to this, however, including the $\omega$-standard models of finite set theory, such as $\<\HF,\in>$, which have no infinite sets at all, and the $\omega$-standard models of set theory with an \emph{amorphous} set. Recall that a set $A$ is \emph{amorphous} in set theory, if it is infinite, but every subset is finite or cofinite in $A$. The existence of such a set is refutable in \ZFC, using the axiom of choice, but it is relatively consistent with \ZF, without the axiom of choice, that amorphous sets exist. Unfortunately, the $\omega$-standard models of set theory with an amorphous set are not $\omega$-saturated, as we explain in observation~\ref{Observation.Amorphous-not-saturated}.

\goodbreak
\begin{theorem}\label{Theorem.Omega-saturated}
 If $\<M,\of^M>$ is a model of set-theoretic mereology arising as the inclusion relation of a model of set theory $\<M,\in^M>$ of any of the following kinds:
 \begin{itemize}
   \item any model of \ZFC\ set theory, or more generally
   \item any model of \KP\ without an amorphous set,
   \item any $\omega$-nonstandard model of \KP, even with amorphous sets, or
   \item any $\omega$-nonstandard model of finite set theory $\ZFfin$,
 \end{itemize}
 then $\<M,\of^M>$ is $\omega$-saturated.
\end{theorem}

\begin{proof}
We are using \KP\ here as a stand-in for any essentially weak set theory, and considerabbly weaker theories suffice. What we require of the model of set theory $\<M,\in^M>$ is first, that it satisfies that the inclusion relation $\of^M$ is an unbounded atomic relatively complemented distributive lattice, and this is truly a very weak requirement on the set theory; and second, that it satisfies the two partition properties stated in theorem~\ref{Theorem.Omega-saturation-criterion}. In any model of \ZFC, for example, every infinite set (including the nonstandard finite sets, if the model is $\omega$-nonstandard) is the union of two disjoint infinite sets and every set has an infinite set disjoint from it. This is also true in \KP\ and even considerably weaker theories, provided that there is no amorphous set, an infinite set which cannot be split into two infinite subsets. If the model of set theory is $\omega$-nonstandard, then the problem of amorphous sets evaporates, since one can use nonstandard finite sets to realize the partition properties. The point is that the saturation criterion of theorem~\ref{Theorem.Omega-saturation-criterion} makes reference to the external concept of infinite elements, whereas amorphous sets make reference to the internal concept of infinite in set theory; in an $\omega$-nonstandard model, therefore, the nonstandard finite sets are infinite with respect to the lattice-theoretic concept, even though they are finite with respect to the internal set-theory of the model. It follows that all $\omega$-nonstandard models of even extremely weak set theories, including finite set theory $\ZFfin$, will have $\omega$-saturated mereological reducts. In summary, for all the kinds of models of set theory listed in the statement of the theorem, the corresponding inclusion relation $\<M,\of^M>$ fulfills the criterion of theorem~\ref{Theorem.Omega-saturation-criterion} and is therefore $\omega$-saturated.
\end{proof}

The import of theorem~\ref{Theorem.Omega-saturated} is that all countable $\omega$-saturated models of a complete theory are isomorphic, by the back-and-forth method. Thus, all of these countable models of set theory have isomorphic inclusion relations.

\begin{corollary}\label{Corollary.All-isomorphic}
 All countable models of \ZFC\ set theory have the same inclusion relation, up to isomorphism. More generally, the inclusion relations arising in any of the models of set theory of the type mentioned in theorem~\ref{Theorem.Omega-saturated} are all isomorphic.
\end{corollary}

\begin{proof}
If $\<M,\in^M>$ is a countable model of set theory of the type mentioned in the statement of the theorem, then the associated mereological structure $\<M,\of^M>$ is a countable $\omega$-saturated model of set-theoretic mereology. Since this is a complete theory, all such models are isomorphic by the back-and-forth construction.
\end{proof}

Let us draw out the consequences answering questions~\ref{Question.Which-models-arise?} and~\ref{Question.What-does-inclusion-know?}.

\begin{corollary}\label{Corollary.Which-models-arise}
 A countable model $\<M,\sqof>$ of set-theoretic mereology arises as the inclusion relation $\of^M$ of a countable model of set theory (of any desired type mentioned in theorem~\ref{Theorem.Omega-saturated}) if and only if it is $\omega$-saturated.
\end{corollary}

\begin{proof}
  All such inclusion relations are $\omega$-saturated, and there is only one countable $\omega$-saturated model of set-theoretic mereology, since it is a complete theory.
\end{proof}

\begin{corollary}\label{Corollary.Alternative-theory}
 If $\<M,\sqof>$ arises as the inclusion relation of a countable model of set theory $\<M,\in^M>$ of the type mentioned in theorem~\ref{Theorem.Omega-saturated}, then for any alternative set theory $T$ extending some extremely minimal theory, there is a relation $\in^*$ on $M$ such that $\<M,\in^*>$ is a model of $T$ and $\<M,\sqof>$ is also the inclusion relation as defined in $\<M,\in^*>$.
\end{corollary}

This is a sense in which the mereological model $\<M,\sqof>$ knows very little set theory, since it cannot tell whether it came from the model $\<M,\in^M>$ or from the model $\<M,\in^*>$, and these can have extremely different theories. Indeed, since the theory $T$ was arbitrary, the original structure $\<M,\sqof>$ is the inclusion relation of a model of any sufficient set theory at all, with \CH, without \CH, with large cardinals, without large cardinals, with infinite sets, or without (but in this case with nonstandard finite sets).

\begin{proof}
If $\<M,\sqof>$ is the inclusion relation of a countable model of set theory $\<M,\in^M>$ of the type mentioned, then $\<M,\sqof>$ is $\omega$-saturated, and hence isomorphic to the inclusion relation of any other countable model of set theory or even countable nonstandard model of finite set theory. By pulling the set-membership relation of that other model back to $M$, we obtain a relation $\in^*$ on $M$ realizing exactly $\sqof$ as its inclusion relation and satisfying that other theory.
\end{proof}

Next, we point out that the omissions of the $\omega$-standard models of finite set theory and the $\omega$-standard models of set theory with amorphous sets in theorem~\ref{Theorem.Omega-saturated} are necessary.

\begin{observation}\label{Observation.Amorphous-not-saturated}
 If $\<M,\in^M>$ is an $\omega$-standard model of finite set theory or an $\omega$-standard model of set theory with an amorphous set, then $\<M,\of^M>$ is not $\omega$-saturated. The inclusion relation for these kinds of models is therefore not isomorphic to the inclusion relations of the other models of set theory mentioned in theorem~\ref{Theorem.Omega-saturated}.
\end{observation}

\begin{proof}
 The $\omega$-standard models of finite set theory have no infinite sets, and therefore fail to fulfill the second criterion of theorem~\ref{Theorem.Omega-saturation-criterion}. If $\<M,\in^M>$ is $\omega$-standard and has an amorphous set $u$, then every subset of $u$ in $M$ is either finite or cofinite in $u$, and so $\<M,\of^M>$ does not realize the type $p(x,u)$ asserting that $x\intersect u$ and $u-x$ are both infinite, although this type is finitely consistent with the theory of $\<M,\of^M>$. So $\<M,\of^M>$ is not $\omega$-saturated.
\end{proof}

We have identified at least three distinct isomorphism types of models of set-theoretic mereology that arise as the inclusion relation of a model of some kind of set theory.
\begin{itemize}
  \item The $\omega$-saturated models arise as $\<M,\of^M>$ for the models of set theory mentioned in theorem~\ref{Theorem.Omega-saturated}.
  \item The prime model, where every set is finite, arises in the $\omega$-standard models of finite set theory $\ZFfin$, such as in $\<\HF,\of>$.
  \item Non-prime non-saturated models $\<M,\of^M>$ arise in an $\omega$-standard model of set theory with an amorphous set.
\end{itemize}

\begin{question}\label{Question.Amorphous-all-isomorphic?}
 Do all countable $\omega$-standard models of\/ \ZF\ set theory with an amorphous set have isomorphic inclusion relations?
\end{question}

Let us expand our understanding of what is possible in the theory of set-theoretic mereology by constructing additional pairwise non-isomorphic countable models of the theory. Fix any natural number $n$ and let $A_0,A_1,\ldots, A_n$ be disjoint infinite sets. Let $M$ consist of sets $A\of\union_i A_i$ such that $A\intersect A_i$ is finite or cofinite for each $i<n$ and $A\intersect A_n$ is finite. This is a collection of sets closed under union, intersection and relative complement; it contains the singleton subsets of any of its members; and it has no largest element. So it is an unbounded atomic relatively complemented distributive lattice and therefore a model of set-theoretic mereology. Notice furthermore that each $A_i$ for $i<n$ is in the collection $M$, and so $M$ has a family of $n$ infinite disjoint sets. But also, every infinite element of $M$ has cofinitely many members from some $A_i$ for $i<n$, and so by the pigeonhole principle there cannot be a family of $n+1$ many infinite disjoint sets in $M$. So $\<M,\of>$ is a model of set-theoretic mereology with this characteristic $n$, the size of the largest family of infinite disjoint sets. Since this characteristic is invariant by isomorphism, we have therefore constructed infinitely many non-isomorphic models of mereology.

Let us observe further that this characteristic, when finite, determines the isomorphism class of the structure. Suppose that $\<M,\sqof>$ is an arbitrary model of set-theoretic mereology with characteristic $n$, so that there are $n$ infinite disjoint sets $A_i$ for $i<n$, but there is no family of $n+1$ many infinite disjoint sets. (We identify every element of $M$ with the set of atoms below it.) In this case, every set in $M$ must agree either finitely or cofinitely with each $A_i$, and contain finitely many additional elements, for otherwise we could construct a family of size $n+1$. Furthermore, every such pattern is realized, since $\<M,\sqof>$ is an atomic relatively complemented distributive lattice, and since it is also unbounded, there must be infinitely many atoms not in any of the $A_i$. It follows that $\<M,\sqof>$ is isomorphic to the model we constructed in the previous paragraph, and so having finite characteristic $n$ is indeed an isomorphism classifier.

The $\omega$-saturated countable model of set-theoretic mereology definitely has infinitely many disjoint infinite sets, but this is not an isomorphism-classifier, on account of the models of \ZF\ with amorphous sets.

The ideas of the previous paragraphs amount to the beginnings of Tarski's classification of the elementary classes of Boolean algebras by means of what are now known as the Tarski invariants~\cite[theorem~5.5.10]{ChangKeisler1990:ModelTheory}, \cite[theorem~6.20]{Poizat2000:ACourseInModelTheory}, \cite[p.~66]{Hodges1993:ModelTheory}. \Ersov\ extended that work to the relatively complemented distributive lattices~\cite{Ershov1964:DecidabilityOfTheElementaryTheoryOfRelativelyComplementedLatticesAndOfTheTheoryOfFilters}, \cite[theorem~15.6]{Monk1976:MathematicalLogic} and~\cite[theorem~3.1.1]{BaudischSeeseTuschikWeese1985:DecidabilityAndQuantifierElimination}, and we expect those invariants to shed light on the classification of the models of set-theoretic mereology arising from models of set theory. In particular, we believe that the answer to question~\ref{Question.Amorphous-all-isomorphic?} will come by means of the Tarski/\Ersov\ invariants of Boolean algebras and relatively complemented lattices applied to the countable models of set-theoretic mereology, combined with knowledge of the models of \ZF\ with amorphous sets. For example, in any model of \ZF\ with an amorphous set $A$, there must be infinitely many pairwise-disjoint amorphous sets, since one may use $\singleton{\alpha}\times A$, and therefore in the quotient of the mereological structure by the ideal generated by the atoms, these will remain as infinitely many atoms in the quotient. It is less clear what happens upon iterating this quotient process, and this would seem to be the key to answering question~\ref{Question.Amorphous-all-isomorphic?}.

\section{Uncountable models of set-theoretic mereology}\label{Section.Uncountable-models}

Let us consider the situation for uncountable models of set-theoretic mereology. Alfredo Dolich pointed out that on general model-theoretic grounds, the analogue of corollary~\ref{Corollary.All-isomorphic} does not hold for uncountable models of set-theoretic mereology:

\begin{theorem}\label{Theorem.2^kappa-many-models}
  For every uncountable cardinal $\kappa$, there are $2^\kappa$ many pairwise non-isomorphic models of set-theoretic mereology arising as the inclusion relation $\of$ in a model of any particular set theory.
\end{theorem}

\begin{proof}
Set-theoretic mereology is an unstable theory, because $\of$ is an order relation, and furthermore, the class of reducts to $\of$ of the models of a given set theory form a pseudo-elementary (PC) class. It follows by deep results of Shelah~\cite[chapter~VIII]{Shelah1990:Classification-theory-and-the-number-of-nonisomorphic-models} that for every uncountable cardinal $\kappa$, there are $2^\kappa$ many non-isomorphic models in that PC class.
\end{proof}

Next, we observe that if one begins with a saturated mereological model, then indeed it does arise as the inclusion relation of a model of set theory.

\begin{theorem}\label{Theorem.Saturated-models-are-realized}
If a (possibly uncountable) model $\<M,\sqof>$ of set-theoretic mereology is saturated, then it arises as the inclusion relation $\of^M$ with respect to a model $\<M,\in^M>$ of any desired consistent set theory.
\end{theorem}

\begin{proof}
Suppose that $\<M,\sqof>$ is a saturated model of set-theoretic mereology, and let $T$ be any consistent set theory, extending \KP, say, for definiteness. It follows that $\<M,\sqof>$ is resplendent (see~\cite[theorem 9.17]{Poizat2000:ACourseInModelTheory}). What this means is that any first-order assertion in the language of a new predicate symbol (that is, a $\Sigma^1_1$ assertion if one understands that one is asserting that there is a predicate satisfying the property) that is consistent with the elementary diagram of the model is realized already by an expansion of the model, adding an actual predicate, but not adding new elements.

Applying resplendency, let us consider the assertion in the language of a new binary relation $\hat\in$, asserting that $\hat\in$ is a model of $T$ and that $\sqof$ is the subset relation with respect to $\hat\in$. This assertion is consistent with the elementary diagram of $\<M,\sqof>$, because a finite subtheory of this theory is asserting merely that $\hat\in$ satisfies $T$ and there are finitely many sets having a certain number of elements in their respective Venn diagram cells. But those assertions are compatible with any model of $T$. Thus, there must be an elementary extension of $\<M,\sqof>$ in which the assertion about $\hat\in$ is realized. So by resplendency, it is already realized without adding any new elements. In other words, there is a relation $\hat\in$ on $M$ such that $\<M,\hat\in>$ is a model of $T$ and $\sqof$ is the inclusion relation $\of$ as defined in it, as desired.
\end{proof}

Conversely, it is easy to see that if $\<M,\in^M>$ is a saturated model of set theory, then the corresponding inclusion model $\<M,\of^M>$ is a saturated model of set-theoretic mereology.

But meanwhile, many uncountable models of set theory do not have saturated mereological inclusion relations.

\begin{theorem}\label{Theorem.Transitive-models-not-saturated}
No uncountable transitive model of set theory $\<M,\in>$ has a saturated inclusion relation $\<M,\of>$. Indeed, if $\<M,\in^M>$ is any model of set theory (a very weak theory suffices) with an element $w\in M$ for which the set of elements $\set{a\in M\mid a\in^M w}$ is countably infinite, then $\<M,\of^M>$ is not $\omega_1$-saturated.
\end{theorem}

\begin{proof}
If $\<M,\in>$ is a transitive model of \ZFC, then $\omega\in M$ is such a set $w$ as in the second statement. So assume that we have a set $w\in M$ with $\set{a\in M\mid a\in^M w}$ being a countably infinite set. Let $p(x)$ be the type asserting that $x\of w$, that $x$ has at least one element, and that $\singleton{a}\not\of x$ for each $a$ with $a\in^M w$. That is, the type $p(x)$ asserts that $x$ is a nonempty subset of $w$, but that it doesn't contain as a subset any particular singleton $\singleton{a}$ of an element of $w$. The type is finitely realized in $\<M,\of^M>$, since we can easily find a subset of $w$ avoiding any particular finite list of elements, but the type cannot be realized in $\<M,\of^M>$, since if $x$ is to be a nonempty subset of $w$, it must contain at least one $\singleton{a}$ for $a\in w$ as a subset. Since the type uses countably many parameters, it follows that $\<M,\of^M>$ is not $\omega_1$-saturated.
\end{proof}

The argument of theorem~\ref{Theorem.Transitive-models-not-saturated} generalizes to other cardinals, showing that if $\<M,\in^M>$ has an element with $\kappa$ many $\in^M$-elements, then $\<M,\of^M>$ is not $\kappa^+$-saturated. Thus, if $\<M,\in^M>$ is a model of set theory whose inclusion relation is $\<M,\of^M>$ is saturated, then all infinite elements of $M$ must have the same cardinality as $M$ itself.

The case of $\omega_1$-like models of set theory is interesting. A model of set theory $\<M,\in^M>$ is \emph{$\omega_1$-like}, if it is uncountable, yet every set $a\in M$ has only countably many $\in^M$-elements. Equivalently, it is uncountable, but every rank initial segment of the model $(V_\alpha)^M$ is countable. It is a generally observed phenomenon for models of arithmetic and set theory that fundamental facts about the countable models often generalize to the $\omega_1$-like models, and when they do not, this is usually interesting. In light of this, it seems natural to ask: are the mereological reducts $\<M,\of^M>$ of the $\omega_1$-like models of set theory $\<M,\in^M>\satisfies\ZFC$ all isomorphic? The answer is no under the $\diamondsuit$ hypothesis.

\begin{theorem}\label{Theorem.Omega1-like-non-isomorphic}
  If $\diamondsuit$ holds and \ZFC\ is consistent, then there is a family of $2^{\omega_1}$ many $\omega_1$-like models of \ZFC\ with pairwise non-isomorphic inclusion relations.
\end{theorem}

\begin{proof}
The main idea is to follow the construction method of~\cite{FuchsGitmanHamkins2017:IncomparableOmega1-likeModelsOfSetTheory}, building a tree of top-extensions of models and using the $\diamondsuit$-sequence to anticipate and then kill off possible isomorphisms of the inclusion relation at each stage.

Assume that $\<A_\alpha\mid\alpha<\omega_1>$ is a $\diamondsuit$-sequence. We shall assign to each transfinite binary sequence $s\in 2^{<\omega_1}$ a countable model $M_s\satisfies\ZFC$ in such a way that they form elementary top-extensions as one lengthens $s$, and each $M_s$ is built on a countable subset of $\omega_1$. We start at the bottom with a given countable model $M_\emptyset$ of \ZFC. At most stages of the construction, including every finite stage and every stage that is not a limit ordinal, if $M_s$ has been defined then we may let $M_{s\concat0}$ and $M_{s\concat 1}$ be arbitrary countable elementary top-extensions of $M_s$ (see~\cite{KeislerMorley1968:ElementaryExtensionsOfModelsOfSetTheory}, proof also given in~\cite[lemma 2]{FuchsGitmanHamkins2017:IncomparableOmega1-likeModelsOfSetTheory}). The interesting case occurs at successors of limit ordinals. If $M_s$ is defined for all $s\in 2^{<\lambda}$, where $\lambda$ is a limit ordinal, we define $M_s$ for $s\in 2^\lambda$ as the union of the elementary chain $M_{s\restrict\alpha}$ for $\alpha<\lambda$. And next, the crucial step. We look at $A_\lambda$ and inquire whether by some unlikely miracle it happens to code a pair of distinct sequences $s,t\in 2^\lambda$ and an isomorphism $j:\<M_s,\of^{M_s}>\cong \<M_t,\of^{M_t}>$, whose underlying sets are contained in $\lambda$. If not, we extend arbitrarily as usual, but if it does, then we shall now extend $M_s$ and $M_t$ in such a way so as to prevent this particular isomorphism from growing.

In order to do so, first extend $M_t$ arbitrarily to $M_{t\concat0}=M_{t\concat 1}$. For each $a\in M_{t\concat0}$, we consider the \emph{trace} of $a$ on $M_t$, which is $\tau(a)=\set{b\in M_t\mid M_{t\concat0}\satisfies b\of a}$. Since there are only countably many new elements $a$, it follows that there are also only countably many trace sets $\tau(a)$. Pulling back under $j$, we may consider the corresponding traces on $M_s$, namely, $j^{-1}\tau(a)=\set{b\in M_s\mid M_{t\concat 0}\satisfies j(b)\of a}$. We claim as in~\cite[lemma 3]{FuchsGitmanHamkins2017:IncomparableOmega1-likeModelsOfSetTheory} that there are continuum many traces on $M_s$ realized in various top-extensions of $M_s$, and so we may find a top-extension of $M_s$ to a model $M_{s\concat 0}=M_{s\concat 1}$ with a new element $c$ whose trace on $M_s$ is different from all those $j^{-1}\tau(a)$. It follows that no further top-extension of $M_{s\concat 0}$ and $M_{t\concat 0}$ to models $M$ and $N$, respectively, can have an isomorphism $j:\<M,\of^M>\cong \<N,\of^N>$ extending $j$, since the trace of $j(c)$ on $M_t$ will have to agree with the trace of some element $a\in M_{t\concat 0}$ on $M_t$, since in $N$ we may take $a=j(c)\intersect V_\alpha^N$ for some rank $\alpha$ between the height of $M_t$ and $M_{t\concat 0}$, and this would mean that the trace of $c$ on $M_s$ would agree with $j^{-1}\tau(a)$, contrary to our choice of $c$. This completes the construction of the models $M_s$ for $s\in 2^{<\omega_1}$. To summarize, we have built a tree of countable top-extensions $M_s$ for $s\in 2^{<\omega_1}$, and at each stage, if the $\diamondsuit$-sequence hands us an isomorphism of the mereological reducts of two models at a given stage, then we extend those models at that stage so as to kill off that isomorphism and prevent it from extending.

To each uncountable binary sequence $S\in 2^{\omega_1}$, let $M_S$ be the union of the elementary chain $M_{S\restrict\alpha}$ for $\alpha<\omega_1$. Thus, we have a family of $2^{\omega_1}$ many $\omega_1$-like models of \ZFC. We claim that the mereological reducts $\<M_S,\of^{M_S}>$ of these models are pairwise non-isomorphic. To see this, assume toward contradiction that $j:\<M_S,\of^{M_S}>\cong\<M_T,\of^{M_T}>$ is an isomorphism for $S\neq T$. We may code $S$, $T$ and $j$ with a subset of $\omega_1$, and so by the $\diamondsuit$ hypothesis, there is a stationary set of $\lambda$ where $A_\lambda$ codes $S\intersect\lambda$, $T\intersect\lambda$ and $j\restrict\lambda$, and where $M_{S\restrict\lambda}$ and $M_{T\restrict\lambda}$ have underlying set contained in $\lambda$. In this case, we had exactly extended these models so as to prevent $j\restrict\lambda$ from extending further, contrary to our assumption that $j$ is an isomorphism of the full models $M_S$ and $M_T$. So there can be no such isomorphism.
\end{proof}

\begin{question}
 Can the $\diamondsuit$ hypothesis in theorem~\ref{Theorem.Omega1-like-non-isomorphic} be omitted?
\end{question}

\section{Class-theoretic mereology}

Let us now extend our analysis of set-theoretic mereology to the case of the various second-order set theories, such as \Godel-Bernays set theory \GBC\ or Kelley-Morse set theory \KM, which allow proper classes as objects in the theory. Although these second-order set theories are commonly presented in a two-sorted language, with one sort for the first-order objects, the sets, and another sort for the second-order objects, the classes, nevertheless both \GBC\ and \KM\ and most of the other second-order set theories also admit one-sorted formalizations, where every object is a class. In that manner of formalism, the sets are simply special kinds of classes, the classes that happen to be an element of another class.

If $\<M,\in>$ is such a model of class theory, then we may define the usual inclusion relation $\of$ on classes and consider the \emph{class-theoretic mereological structure} $\<M,\of>$, where we keep all the classes of $M$ but now have only the inclusion relation. It is easy to see that $\<M,\of>$ is an atomic Boolean algebra with infinitely many atoms. This is a complete theory by classical results of Tarski, and because of how we arrived at this theory, we shall refer to it as the theory of class-theoretic mereology.

Our main result for these structures is that the class-theoretic mereological reducts $\<M,\of^M>$ are all $\omega$-saturated, and therefore all countable such models are isomorphic (theorem~\ref{Theorem.Class-theory-omega-saturated} and corollary~\ref{Corollary.Class-theory-all-isomorphic}).

Tarski provided an elimination-of-quantifiers construction, showing that every assertion in the language of Boolean algebras is equivalent in class-theoretic mereology to a quantifier-free assertion in the language of Boolean algebras augmented by the relations $|\tau|=n$, which asserts that object $\tau$ has precisely $n$ atoms below it. Indeed, that quantifier-elimination argument was the inspiration for the argument we had given in~\cite{HamkinsKikuchi2016:Set-theoreticMereology} for the case of set-theoretic mereology. It follows that we get an analogue of lemma~\ref{Lemma.Types} for class-theoretic mereology.

\begin{lemma}\label{Lemma.Class-types}
 If $p(a_1,\ldots,a_n)$ is a complete type with $n$ free variables in the language of mereology, then $p(a_1,\ldots,a_n)$ is equivalent over the theory of class-theoretic mereology to the assertions stating for each cell in the Venn diagram of the variables (including the universal class, providing the exterior region) that it has some specific finite size or that it is infinite.
$$
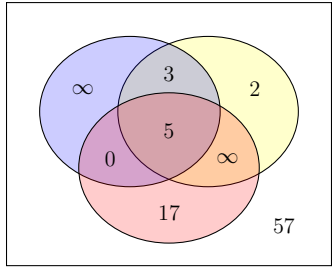
\begin{tikzpicture}[xscale=1.2,every node/.style={scale=.7}]
\draw (-1.8,-1.8) rectangle (1.8,1.7);
\draw[fill=yellow, fill opacity=.2] (30:.5) circle (1);
\draw[fill=blue, fill opacity=.2] (150:.5) circle (1);
\draw[fill=red, fill opacity=.2] (-90:.5) circle (1);
\draw (0,0) node {$5$};
\draw (30:1.1) node {$2$};
\draw (150:1.1) node {$\infty$};
\draw (-90:1.1) node {$17$};
\draw (90:.75) node {$3$};
\draw (-30:.75) node {$\infty$};
\draw (210:.75) node {$0$};
\draw (-45:1.8) node {$57$};
\end{tikzpicture}$$
\end{lemma}

One difference between the type assertions here and the case of set-theoretic mereology is that in set-theoretic mereology, the exterior region was always infinite, since one of the axioms was that the lattice is unbounded; but in class theory, there is a universal class $V$, and one can have cofinite proper classes and so on.

\begin{proof}
The proof is essentially just the same as for lemma~\ref{Lemma.Types}. If we have a complete type $p(a_1,\ldots,a_n)$, then it will assert particular values for those cells in the Venn diagram, and the point is that this information completely determines the rest of the type by the quantifier-elimination result.
\end{proof}

\begin{theorem}\label{Theorem.Class-theory-omega-saturated}
 If $\<M,\in^M>$ is a model of \GBC, but considerably less suffices, then the corresponding inclusion relation $\<M,\of^M>$ is an $\omega$-saturated model of class-theoretic mereology, an $\omega$-saturated infinite atomic Boolean algebra.
\end{theorem}

\begin{proof}
We can prove this theorem in essentially the same manner as we proved theorem~\ref{Theorem.Omega-saturated}. Given any complete type $p(x,\vec a)$ that is finitely realized in $\<M,\of^M>$, we know by lemma~\ref{Lemma.Class-types} that $p$ is asserting that $x$ exhibits a certain pattern of sizes for the cells in the Venn diagram of the parameters. But any model of class theory is able to realize any such finite pattern, just as before, provided that we are not in an $\omega$-standard model of finite set theory or an $\omega$-standard model with an amorphous set.
\end{proof}

\begin{corollary}\label{Corollary.Class-theory-all-isomorphic}
 All countable models of \GBC\ have the same inclusion relation, up to isomorphism. Specifically, if $\<M,\in^M>$ and $\<N,\in^N>$ are each countable models of \GBC, then $\<M,\of^M>$ is isomorphic to $\<N,\of^N>$.
\end{corollary}

\begin{proof}
Since $\<M,\of^M>$ and $\<N,\of^N>$ are each $\omega$-saturated models of the same complete theory, they are isomorphic by the back-and-forth construction.
\end{proof}

The class theory required in these theorems is extremely weak. All that is needed is to prove that $\of$ forms an infinite atomic Boolean algebra and the ability to realize types of the limited expressive power identified in lemma~\ref{Lemma.Class-types}. For example, one can make an $\omega$-saturated model by starting even with a nonstandard model of finite set theory $\ZFfin$ and adding as classes all the definable classes. One has a natural theory $\GBC^{\neg\infty}$ corresponding to this situation, and this structure could be thought of as a class-theoretic analogue of finite set theory, but it still is sufficient to support the $\omega$-saturated argument of theorem~\ref{Theorem.Class-theory-omega-saturated}. For example, the model $\<\HF,\in>$ augmented with its definable classes is isomorphic in the inclusion $\of$ reduct to the inclusion reduct of any countable model of full \GBC. We find this remarkable.

Let us mention specifically, however, that one must again exclude the $\omega$-standard models of class theory having an amorphous set, as in this case, the corresponding mereological structures are not $\omega$-saturated. 

Apart from the amorphous exception, the general conclusion again is that the inclusion relation $\of$ of class theory knows very little about the theory of $\in$ in the model from which it arose. Every model of class-theoretic mereology $\<M,\of^M>$ arising as the inclusion relation of some model of class theory also arises identically as the inclusion relation of other models of class theory, with totally different theories, and indeed any given model of class theory is isomorphic to a model giving rise identically to $\of^M$. So if all you know is the inclusion relation $\of$, you cannot tell whether the model of class theory had the continuum hypothesis, whether it had large cardinals, or indeed whether it though the axiom of infinity was true.

Indeed, it follows from theorem~\ref{Theorem.Class-theory-omega-saturated} that for a countable model of class theory $\<M,\in>$, there is no difference between a proper class with infinite complement and an infinite set with a proper class complement, with respect to properties in the corresponding mereological structure $\<M,\of^M>$. For example, one may consider an infinite set of natural numbers in the model or the class of all ordinals in the model, and in the language of inclusion $\of$ both of these classes realize the same parameter-free type by lemma~\ref{Lemma.Class-types}, and so there is an automorphism of $\<M,\of^M>$ swapping them. One cannot distinguish between any two infinite co-infinite classes in mereology, even if one of them begins as a proper class and the other begins as a mere infinite set. Thus, automorphisms of the inclusion relation $\<M,\of^M>$ need not respect the set/class distinction, and in this sense the set/class distinction is not expressible in class-theoretic mereology.

Much of the rest of our treatment of set-theoretic mereology also extends to class-theoretic mereology. For example, although there is a unique countable model arising as the inclusion relation from a model of class theory, nevertheless for uncountable cardinals $\kappa$ there will be $2^\kappa$ many non-isomorphic models of size $\kappa$ arising as the inclusion relation $\of$ of a model of any given class theory. And the inclusion relation of an uncountable transitive model of \GBC\ is never saturated, nor even $\omega_1$-saturated.

Let us conclude the paper by mentioning briefly David Lewis's extended philosophical treatment of the mereological content of class theory in his book~\cite{Lewis1991:PartsOfClasses}. Because his approach gives a central role to the singleton operator $a\mapsto\singleton{a}$, our perspective is that it is consequently closer to the class theory of \Godel-Bernays or Kelley-Morse than it is to the purely mereological theory of inclusion that we consider here. After all, as we pointed out in~\cite{HamkinsKikuchi2016:Set-theoreticMereology}, the $\of$ relation when augmented with the singleton operator becomes inter-definable with the membership relation $\in$. In our previous article~\cite{HamkinsKikuchi2016:Set-theoreticMereology}, we had argued that the decidability of set-theoretic mereology, the pure theory of $\of$, is an important part of the explanation why the $\of$-only form of mereology has not provided a robust foundation of mathematics, and an essentially similar argument applies to class-theoretic mereology, since $\of$ for classes is the theory of an infinite atomic Boolean algebra, again a decidable theory. We take the main results of this article, that there is essentially only one countable model of mereology that arises, to further buttress this argument. Meanwhile, by adopting the singleton operator, Lewis side-steps both of these criticisms, for mereology with the singleton operator is fully bi-interpretable with membership-based set and class theory, if at the cost of being less mereological.

\bibliographystyle{alpha}
\bibliography{HamkinsBiblio,MathBiblio,WebPosts}

\end{document}